\documentclass[a4paper,twoside,11pt]{amsart}

\usepackage{amsthm}

\usepackage{mathrsfs}
\usepackage{array}
\usepackage{graphicx}
\usepackage{amssymb}
\usepackage{amsmath}
\usepackage{color}
\usepackage{tikz}
\usetikzlibrary{arrows,matrix,positioning}

\def\adots{\mathinner{\mkern2mu\raise 1pt\hbox{.}\mkern 3mu\raise
4pt\hbox{.}\mkern1mu\raise 7pt\hbox{{.}}}}

\newtheorem{theorem}{Theorem}[section]
\newtheorem{lemma}{Lemma}[section]

\newtheorem{proposition}{Proposition}[section]

\theoremstyle{remark}

\numberwithin{equation}{section}

\def\cddots{\mathinner{\mkern1mu\raise1pt\vbox{\kern1pt\hbox{.}}\mkern2mu
    \raise4pt\hbox{.}\mkern2mu\raise7pt\hbox{.}\mkern1mu}}

\begin{document}

\title[Finite Blaschke Products with Preassigned Critical Points]
{A Numerical Model for the  Construction\\of Finite Blaschke Products with \\ 
	Preassigned Distinct Critical Points}
\date{}
\author[C.\,Glader]{Christer Glader}
\email{cglader@abo.fi}
\address{Department of Mathematics, \AA bo Akademi University, FIN-20500,
 \AA bo, Finland}

\author[R.\,P\"orn]{Ray P\"orn}
\email{rporn@abo.fi}
\address{Department of Didactics of Mathematics, \AA bo Akademi University, FIN-65100,
 Vasa, Finland.}



\thanks{The first author was supported by the Magnus Ehrnrooth Foundation.}

\begin{abstract}
We present a numerical model for determining a finite Blaschke
product of degree $n+1$ having $n$ preassigned distinct critical
points $z_1,\dots,z_n$ in the complex (open) unit disk
$\mathbb{D}$. The Blaschke product is uniquely determined up to
postcomposition with conformal automorphisms of $\mathbb{D}$.
The proposed method is based on the construction of a sparse nonlinear
system where the data dependency is isolated to two vectors and on a certain transformation of the critical points.
The efficiency and accuracy of the method is illustrated in several examples.
\end{abstract}

\maketitle

\section{Introduction}

A finite Blaschke product of degree $n$ is a rational function of the form
\begin{equation} \label{bladef1}
B(z) = c\, \prod_{j=1}^n \frac{z - \alpha_j}{1 - \overline\alpha_j\, z},\quad
  c,\alpha_j\in\mathbb{C},\ \vert c\vert = 1,\ |\alpha_j| < 1\, ,
\end{equation}
which thereby has all its zeros in the open unit disc
$\mathbb{D}$, all poles outside the closed unit disc
$\overline{\mathbb{D}}$ and constant modulus $|B(z)| = 1$ on the
unit circle $\mathbb{T}$. The overbar in (\ref{bladef1}) and in
the sequel stands for complex conjugation.

The finite Blaschke products of degree $n$ form a subset of the rational
functions of degree $n$ which are unimodular on $\mathbb{T}$.
These functions are given by all fractions
\begin{equation} \label{bladef2}
\tilde B(z) = \frac{a_0 +  a_1\, z + ... +  a_n\, z^n} {\overline
 a_n + \overline a_{n-1}\, z + ... + \overline a_0\, z^n },\quad
 a_0,...,a_n\in {\mathbb{C}}\, .
\end{equation}
An irreducible rational function of form (\ref{bladef2}) is a
finite Blaschke product when all its zeros are in $\mathbb{D}$. We
adopt the convention of calling irreducible rational functions of
form (\ref{bladef2}) Blaschke forms if the function has at least
one pole in $\mathbb{D}$. A Blaschke form can be interpreted as a
quotient of finite Blaschke products.

The critical points of a finite Blaschke product are the zeros of
the derivative and there are exactly $n-1$ critical points in
$\mathbb{D}$, counting multiplicity, if the Blaschke product is of
degree $n$. The classical Gauss-Lucas theorem states that the
critical points of a polynomial lie in the convex hull of its
zeros, see e.g. \cite{ShS} p. 25. The analogous geometrical result
for finite Blaschke products, due to Walsh, is that the critical
points situated in $\mathbb{D}$ lie within or on the non-Euclidean
convex hull of the zeros of the Blaschke product with respect to
the Poincar\'e metric, see \cite{W1} and \cite{W2} p. 157, where
Walsh refers to finite Blaschke products as non-Euclidean
polynomials. In \cite{W1} Walsh also shows that if the zeros
$\alpha_1,\dots,\alpha_n$ of the Blaschke product lie within a
circle $C$ in $\mathbb{D}$, then so do the $n-1$ critical points.
It can furthermore be shown that the $n-1$ critical points in $\mathbb{D}$ lie in
the convex hull of the point
 set $\{0, \alpha_1,\dots,\alpha_n\}$, see \cite{ShS} p. 374.

This paper presents a numerical method to solve the following problem:
\medskip

\noindent {\bf Problem I}. {\it Given $n$ distinct points $z_1,\ldots , z_n$
in $\mathbb{D}$, find a finite Blaschke product $B$ of degree
$n+1$ such that $B'(z_j) = 0,\ j=1,\dots,n$.}
\medskip

A solution to Problem I always exists by a theorem of Heins
\cite{H} p. 48:

\begin{theorem}\label{thmHeins} Given points $z_1,\dots,z_n$ in $\mathbb{D}$, not necessarily distinct, then there is a finite Blaschke product $B$ of degree $n+1$ whose set of critical points in $\mathbb{D}$ coincides with the prescribed points $z_1,\dots,z_n$.
\end{theorem}

The finite Blaschke product $B$ in Theorem~\ref{thmHeins} is
unique up to postcomposition by conformal automorphisms of
$\mathbb{D}$, see the comments on this by Kraus and Roth in
\cite{KR1} on p. 185, where they also note that all proofs that
exist of Theorem~\ref{thmHeins}, e.g. the proof by Zakeri \cite{Z}
and including the one by Heins, are nonconstructive. In \cite{KR1}
a constructive method to find a Blaschke product with prescribed
critical points is presented, but the authors however claim that
it is not really suitable to obtain a solution in explicit form
and pose the question: ``Is there a finite algorithm that allows
one to compute a finite Blaschke product from its critical
points?'' It is primarily this question that has motivated us to
investigate the problem in the case with distinct prescribed
critical points. In \cite{KR1} Theorem~\ref{thmHeins} is
generalized to the (infinite) Blaschke product setting where the
sequence of critical points in $\mathbb{D}$ is required to satisfy
a Blaschke condition. In the recent paper \cite{SW} Semmler and Wegert show that the problem of determining a finite Blaschke product with prescribed critical points is equivalent to two other problems, namely the problem of finding the equilibrium position of moveable point charges interacting with a special configuration of fixed charges, and the problem of solving a moment problem for the canonical representation of power moments on the real axis. Problem I has also been studied in the circle packing context by Stephenson in \cite{Ste}, p. 268.

The {\it Wronskian} of two polynomials $p$ and $q$ that define a rational function $R := p/q$ is given by
\begin{equation*}
W(z) := p'(z)\, q(z) - p(z)\, q'(z)\, ,
\end{equation*}
so $R$ has a critical point $z_i$ of multiplicity $n_i$ if and
only if that critical point is a zero of $W$ of multiplicity
$n_i$. General rational functions can be ordered into classes
where the functions are identical up to postcomposition by
linear-fractional transformations and where the theory of Fuchsian
differential equations and Wronskians come into play, see
\cite{Sch}. The Wronskians of finite Blaschke products
(\ref{bladef1}) and Blaschke forms (\ref{bladef2}) are very
special in structure, they are, after a unimodular scaling, self-inversive polynomials, which
means that if $W$ is of degree $n$, then $e^{i\varphi}\, W(z) = z^n\,
\overline{e^{i\varphi}\, W(1/\overline z )}$, for some $\varphi\in\mathbb{R}$ . More specifically, if $z_1,\dots ,
z_k$ are non-zero critical points with multiplicities $n_1,\dots
,n_k$ of a Blaschke product or a Blaschke form in $\mathbb{D}$ of
degree $n+1$, then also $1/\overline z_1,\dots,1/\overline z_k$
are critical points with corresponding multiplicities. If $n_0$ is
the multiplicity of $z_0=0$ as a critical point, ($n_0 = 0$ if
zero is not a critical point), we obtain
\begin{equation}\label{blaWronskian1}
W(z) = r\, e^{i\theta}\, z^{n_0}\, \prod_{j=1}^k (z-z_j)^{n_j}
(z-\frac{1}{\overline z_j})^{n_j}\, ,\quad r > 0\, ,\ \theta\in\mathbb{R}\, ,\ \sum_{j=0}^k n_j = n\, ,
\end{equation}
as the factored representation in the critical points of the
Wronskian for a finite Blaschke product or Blaschke form of
irreducible degree $n+1$. If the unimodularly scaled Wronskian is written in coefficient form $e^{i\varphi}\, W(z)=\sum_{j=0}^{2n} c_j z^j$, the self-inversiveness implies that $c_j=\overline{c}_{2n-j}$, $j=0,\ldots ,2n$, shortly denoted $c = \mathbf{flip} \, \overline{c}$, where $c$ is the vector of coefficients and \textbf{flip} is the operator that reverses the order of the elements in a vector. We call a vector $c$ self-reversive if it satisfies $c = \mathbf{flip} \, \overline{c}$.
The \textbf{flip} operator applied to a matrix conjugates and reverses
the order of the elements along both dimensions. 

In section 2 we derive a model with linear and quadratic constraints
for numerical solution of Problem I. The efficiency of the method
is demonstrated in section 3.

\section{A sparse linear model for Problem I}

The goal of this section is to recast a nonlinear (conjugate
quadratic) system obtained from the Wronskian for the solution of
Problem I. To this end we derive a sparse linear system with simple
quadratic constraints (\ref{sparsemodel}), that contains the
solution to Problem I and also the meromorphic Blaschke form
solutions. We are able to completely describe the nullspace of a relaxed 
linear system derived from the full Wronskian, apart from one data dependent basis vector which,
together with a particular solution to the system, can be computed
with the fast Fourier transform. This renders a good starting point for the solution
scheme described in section 3 for solving the reduced relaxed linear system with quadratic constraints (\ref{sparsemodel}). To accomplish the
derivation of the model (\ref{sparsemodel}) reordering of the terms of the Wronskian
are performed according to their polynomial degree and certain index sets are constructed. 
To make this process more transparent we illustrate it with intertwined examples for the case $n = 3$.

To singel out a particular Blaschke product $B$ of degree $n+1$
that solves Problem I above, one could use the normalization by
Zakeri in \cite{Z}: $B(0) = 0$ and $B(1) = 1$. We however proceed
with a different normalization presented below and assume that $n$
distinct nonzero critical points $z_1,\ldots,z_n$ in $\mathbb{D}$
are preassigned. 

That $z_i\ne 0$ introduces no loss of
generality, because if $z_1=0$ we could transform the data by
defining 
\begin{equation}\label{transformation}
b(z)=\frac{z-z_\star}{1-\overline{z}_\star z}
\end{equation}

with $z_\star\in\mathbb{D}$ and $z_\star\ne z_j\, ,\ j=1,\ldots ,n$. 
Later, in section 3, we in fact advocate the use of $z_\star=(z_1+\ldots+z_n)/n$
 as a standard normalization of the data. Then, if $\tilde B (z)$ is a finite Blaschke
 product of degree $n+1$ with critical points $ b(z_1),\ldots,b(z_n)$, the
Blaschke product $B(z) := \tilde{B}(b(z))$ solves Problem I for
$z_1,\ldots,z_n$. Having computed $\tilde B$ and its zeros $\tilde\alpha_j$, the
 zeros $\alpha_j$ of the Blaschke product $B$ are given by $\alpha_j = b^{-1}(\tilde\alpha_j)$.

Define $z_{n+1} = 1/\overline
z_1,\dots, z_{2n}=1/\overline z_{n}$. Then we have a vector
$[z_1,\dots,z_{2n}]^T$ of $2 n$ distinct prescibed critical
points, of which the $n$ first are in $\mathbb{D}$ and the rest
are outside $\overline{\mathbb{D}}$.
Now we make the Ansatz that
the Blaschke product $B$ is of degree $n+1$ and of form
(\ref{bladef2}) with the normalization $B(0) = 0$ and $a_{n+1} =
1$,
\begin{equation}\label{Ansatz}
B(z) = \frac{p(z)}{q(z)}=\frac{a_1 z + \dots + a_n\, z^n + z^{n+1}}{1 + \overline a_n\, z + \dots + \overline a_1\, z^n}\, .
\end{equation}
For this Ansatz the requirement $W(z_k) = 0,\ k=1,\dots, 2n$,
generates a nonlinear system in the complex coefficients $a_i$. Expansion of
\begin{align*}
W(z) =\ & p'(z)q(z)-p(z)q'(z)  \\
     =\ &(a_1+2a_2z+...+na_nz^{n-1}+(n+1)z^n)\cdot (1+\overline{a}_nz+...+\overline{a}_1z^n) \\
        &- (a_1z +a_2z^2+ ... + a_nz^n + z^{n+1})\cdot (\overline{a}_n + 2\overline{a}_{n-1}z + ... + n\overline{a}_1z^{n-1}) \\
     =\ &\sum_{i=1}^n ia_iz^{i-1} + \sum_{i=1}^n \sum_{j=1}^n ia_i\overline{a}_{n-j+1}z^{i+j-1} +
                    (n+1)\sum_{j=1}^n \overline{a}_{n-j+1}z^{n+j}\\
     &+\, (n+1)z^n - \sum_{i=1}^n \sum_{j=1}^n ja_i\overline{a}_{n-j+1}z^{i+j-1} - \sum_{j=1}^n j\overline{a}_{n-j+1}z^{n+j}\, ,
\end{align*}
gives

\begin{equation}\label{Wronskian1} 
 \begin{aligned}  
     W(z) =\ &\sum_{i=1}^n \sum_{j=1}^n (i-j) a_i\overline{a}_{n-j+1} z^{i+j-1} +
                        \sum_{i=1}^n i a_i z^{i-1}\\
                 &+ \sum_{i=1}^n (n-i+1) \overline{a}_{n-i+1} z^{n+i} + (n+1)z^n\, .
  \end{aligned}                      
\end{equation}

\noindent Thus our focus will be on the following nonlinear system

\begin{equation}\label{System1}
\begin{aligned} &\underbrace{\sum_{i=1}^n \sum_{j=1}^n (i-j)
a_i\overline{a}_{n-j+1} z_k^{i+j-1}}_{\textrm{quadratic in } a} +
\underbrace{\sum_{i=1}^n i a_i z_k^{i-1} + \sum_{i=1}^n (n-i+1)
\overline{a}_{n-i+1} z_k^{n+i}}_{\textrm{linear in } a} \\
& \qquad\quad + \underbrace{ (n+1) z_k^n}_{\textrm{constant}} =
0\qquad \iff \qquad A x(a) = b\, ,
\end{aligned}
\end{equation}
where $A$ is a matrix of size $2n\times (n^2+n)$, $b=-(n+1)
[z_1^n,\dots,z_{2n}^n]^T$ and $x = x(a)$ is a vector of variables
with quadratic structure. This dense system has $n$ variables, $2n$ equations and each equation has $n^2-n$ conjugate quadratic terms and $2n$ linear terms. It turns out that the matrix $A$ has
very high condition number, mainly due to many high powers of $z_k$, 
resulting in an ill-conditioned system. This problem can be circumvented by studying the {\bf relaxed
linear system}
\begin{equation}\label{Relaxedsystem}
  A x = b\, ,\quad x\in\mathbb{C}^{n^2+n}\, ,
\end{equation}
which contains all solutions to the nonlinear system $A x(a) = b$.
When the quadratic structure is considered the vector is denoted by $x(a)$ and
if this condition is relaxed the vector is denoted by $x$.

The Wronskian (\ref{Wronskian1}) has a total of
$n^2+n$ terms that are dependent on coefficients $a$ and one
constant term. The terms in $W$ are now reordered. All terms that
correspond to a certain degree $d$, $0\leq d\leq 2n$, are summed
in a specific order. For $0\leq d\leq n-1$ each sum starts with the linear term of degree
$d$ and continues with all quadratic terms of degree $d$ ordered
with increasing index $i$ for variables $a_i$. 
For $n+1\leq d\leq 2n$ the order is reversed.
The part of the Wronskian that contains
all terms of degree $d$ dependent of $a$ is denoted by $W^d(z)$, so
\begin{equation}\label{WronskDef}
W(z)=\sum_{d=0}^{2n} W^d(z) + (n+1)z^n.
\end{equation}

The formal definition of $W^d(z)$ is:

\begin{equation}\label{Wronskian2}
W^d(z)= 
 \begin{cases}
         a_1\, ,   &   \text{if } d=0, \\
         2a_2z\, , & \text{if } d=1, \\
         (d+1)a_{d+1}z^d + \\ 
         \quad\qquad \sum_{i=1}^d (2i-d-1)a_i\overline{a}_{n-d+i}z^d\, , & \text{if } 2\leq d\leq n-1, \\
                \sum_{i=1}^n (2i-n-1)a_i\overline{a}_iz^n\, ,  & \text{if } d=n, \\
              \sum_{i=1}^{2 n-d} (2n-2i-d+1)\overline{a}_{2n-d-i+1} a_{n-i+1}z^d+ \\  \quad\qquad (2n-d+1)\overline{a}_{2n-d+1}z^d\, ,  & \text{if } n+1\leq d\leq 2n-2, \\
                2\overline{a}_2z^{2n-1}\, , &   \text{if } d=2n-1, \\
                \overline{a}_1z^{2n}\, , &  \text{if } d=2n\, .
 \end{cases}
\end{equation}
\medskip

\begin{proposition} The Wronskian given in equation (\ref{Wronskian1}) is equivalent to the definition $W(z)=\sum_{d=0}^{2n} W^d(z)+(n+1)z^n$ in (\ref{WronskDef}).
\end{proposition}

\begin{proof}
We start by inspecting all constant and linear terms in the sum (\ref{Wronskian1}):
\begin{equation*}
\sum_{i=1}^n i a_i z^{i-1} + \sum_{i=1}^n (n-i+1) \overline{a}_{n-i+1} z^{n+i} + (n+1) z^n\, .
\end{equation*}
The terms of this sum are written out in increasing order with respect to the degree of $z$ giving
\begin{equation*}
	a_1 + 2a_2z + \sum_{i=3}^n i a_i z^{i-1} + (n+1) z^n + \sum_{i=1}^{n-2} (n-i+1) \overline{a}_{n-i+1} z^{n+i} + 2\overline{a}_2 z^{2n-1} + \overline{a}_1z^{2n}\, .
\end{equation*}
Reindexing the sums results in the expression
\begin{equation*}
	a_1 + 2a_2z + \sum_{i=2}^{n-1} (i+1) a_{i+1} z^i + (n+1) z^n + \sum_{i=n+1}^{2n-2} (2n-i+1) \overline{a}_{2n-i+1} z^i + 2\overline{a}_2 z^{2n-1} + \overline{a}_1z^{2n}\, .
\end{equation*}
This sum is now identical to the sum of all linear terms in $W^d(z)$, $d=0,\ldots,2n$, and the
constant term $(n+1)z^n$.
Next quadratic terms are considered. The quadratic terms appear in three different parts in the definition of $W^d(z)$. We prove only the middle case where $d=n$, the two other cases are proved similarly. All quadratic terms in the sum (\ref{Wronskian1}) that correspond to degree $d=n$ are collected. We set $d=n=i+j-1\Leftrightarrow j=n-i+1$ in the double sum in (\ref{Wronskian1}) obtaining

\begin{align*}
 \sum_{i=1}^n \sum_{j=1}^n (i-j) a_i\overline{a}_{n-j+1} z^{i+j-1} & = 
 \sum_{i=1}^n (i-(n-i+1)) a_i\overline{a}_{n-(n-i+1)+1} z^{i+n-i+1-1} \\
 & =  \sum_{i=1}^n (2i-n-1) a_i\overline{a}_{i} z^n\, ,
\end{align*}
where the last expression is the quadratic sum in $W^d(z)$ corresponding to degree $d=n$. The cases when $2\leq d\leq n-1$ and $n+1\leq d\leq 2n-2$ are proved in a similar way, so the two definitions of $W(z)$ are equivalent.
\end{proof}

For future use we express $W^d(z)$, $2 \leq d\leq 2n-2,$ as matrix products 
\begin{equation*}
W^d(z)= w^d x^d(a)^T z^d\, ,
\end{equation*}
where

\begin{equation}\label{wddef}
\begin{aligned}
  &w^0 = w^{2n} = 1\, , w^1=w^{2n-1}=2\, , k_d := (2n-d+1)/2 \quad \text{and}\\ 
  &w^n =
  \begin{cases}
    \bigl( (2 i-n-1)_{i=1}^n\bigr)\, , & n \text{ even\, ,}\\
    \bigl( (2 i-n-1)_{i=1,\ i\ne (n+1)/2}^n\bigr)\, , & n \text{ odd\, ,}
  \end{cases}\\
  &w^d =
  \begin{cases}
    \bigl(d+1, (2i-d-1)_{i=1}^d\bigr)\, , & d \text{ even} \, ,\ 2\le d\le n-1\, ,\\
    \bigl(d+1, (2i-d-1)_{i=1,\ i\ne (d+1)/2}^d\bigr)\, , & d \text{ odd}\, ,\ 2\le d\le n-1\, ,
  \end{cases}\\
  &w^d =
  \begin{cases}
    \bigl( (2(n-i)-d+1)_{i=1}^{2n-d},2n-d+1\bigr)\, , & d \text{ even} \, ,\ n+1\le d\le 2n-2\, ,\\
    \bigl( (2(n-i)-d+1)_{\genfrac{}{}{0pt}{}{i=1,\, }{i\ne k_d}}^{2n-d},2n-d+1\bigr)\, ,& d \text{ odd}\, ,\ n+1\le d\le 2n-2\, 
  \end{cases}\\
  &x^0(a) =a_1\, ,\ x^1(a)=a_2\, ,\ x^{2n-1}(a)=\overline{a}_2\, ,\ x^{2n}(a)=\overline{a}_1\quad \text{and}\\ 
  &x^n(a) =
  \begin{cases}
    \bigl( (a_i\overline{a}_{i})_{i=1}^n\bigr)\, , & n \text{ even\, ,}\\
    \bigl( (a_i\overline{a}_{i})_{i=1,\ i\neq (n+1)/2}^n\bigr)\, , & n \text{ odd\, ,}
  \end{cases}\\
   &x^d(a) =
  \begin{cases}
    \bigl(a_{d+1},(a_i\overline{a}_{n-d+i})_{i=1}^d\bigr)\, , & d \text{ even} \, ,\ 2\le d\le n-1\, ,\\
    \bigl(a_{d+1},(a_i\overline{a}_{n-d+i})_{i=1,\ i\ne (d+1)/2}^d\bigr)\, , & d \text{ odd}\, ,\ 2\le d\le n-1\, ,
  \end{cases}\\
    &x^d(a) =
  \begin{cases}
    \bigl((\overline{a}_{2n-d+1-i} a_{n+1-i})_{i=1}^{2n-d},\overline{a}_{2n-d+1}\bigr)\, , & d \text{ even} \, ,\ n+1\le d\le 2n-2\, ,\\
    \bigl((\overline{a}_{2n-d+1-i} a_{n+1-i})_{i=1,i\ne k_d}^{2n-d},\overline{a}_{2n-d+1}\bigr)\, , & d \text{ odd}\, ,\ n+1\le d\le 2n-2\, .
  \end{cases}\\
\end{aligned}
\end{equation}

The complete variable and weight vectors, $x=x(a)$ and $w$, are then obtained by
concatenation as $x(a)=(x^0(a),x^1(a),\ldots, x^{2n}(a))$ and 
$w=(w^0, w^1, \ldots, w^{2n})$.  
From (\ref{wddef}) it is clear that $w^d = \mathbf{flip} \, w^{2n-d},\, d\neq n$ and 
$w\backslash w^n = \mathbf{flip} \, w\backslash w^n$ and $x^d(a) = \mathbf{flip} \,  \overline{x^{2n-d}(a)}$ and 
$x(a)\backslash x^n(a) = \mathbf{flip} \, \overline{x(a)\backslash x^n(a)}$.
An upper index $x^d$ refers to the subvector that corresponds to degree $d$ and a 
lower index $x_i$ to the element at position $i$.
The next step is to analyze the order of the elements in the vector $x(a)$. That is, the relationship between variable
$x_i$ and the product $a_j\overline{a}_k$.  We let a zero index indicate that a certain factor is missing and define $a_0:=1$ and
$\overline{a}_0:=1$, so then, for example, $x_1=a_1
\overline{a}_0=a_1$ and $x_{4}=a_1\overline{a}_2$. Index vectors
$I^d$ for variables $a_j$, $0\le d\le n-1$, are obtained directly from (\ref{Wronskian2}),
\begin{equation}\label{Id}
I^d =
\begin{cases}
(1)\, , & d=0\, , \\
\bigl(d+1, (i)_{i=1}^d\bigr)\, , & d \text{ even}\, , \quad 1\leq d\leq n-1,\\
\bigl(d+1, (i)_{i=1}^d\bigr)\backslash  \left(\frac{d+1}{2}\right)\, , & d \text{ odd}\, , \quad\; 1\leq d\leq n-1,
\end{cases} 
\end{equation}

and corresponding vectors for conjugate variables
$\overline{a}_k$,  $0\le d\le n-1$, are 
\begin{equation}\label{conjId}
\overline{I}^d =
\begin{cases}
(0)\, , & d=0\, , \\
\bigl(0, (n-d+i)_{i=1}^d\bigr)\, , & d \text{ even}\, , \quad 1\leq d\leq n-1,\\
\bigl(0, (n-d+i)_{i=1}^d\bigr) \backslash \left(\frac{2n-d+1}{2}\right)\, , & d \text{ odd}\, , \quad\; 1\leq d\leq n-1,
\end{cases}
\end{equation}

and for degree $n$ we have identical index vectors
\begin{equation}\label{Idn}
I^n = \overline{I}^n =
\begin{cases}
(1,2,...,n)\, , & n \text{ even}\, ,\\
(1,2,...,n)\backslash \left(\frac{n+1}{2}\right)\, , & n \text{ odd}\, .
\end{cases}
\end{equation}

Index vectors for degrees $n+1\le d\le 2n$ are given by the $ \mathbf{flip}$-operation.

\begin{lemma} \label{Revlemma}
It holds that $I^d = \mathbf{flip} \, \overline{I}^{2n - d}$ and $\overline{I}^d =  \mathbf{flip} \, I^{2n - d}$ for
 $n+1\le d\le 2n$.   
\end{lemma}

\begin{proof}
It follows directly from the two first and last equations of (\ref{Wronskian2}) that the formulas hold for $d=2 n$ and $d=2n -1$. Let $d=n+k$ for some $k$, $1\le k\le n-2$, and suppose that $d$ is even. Then $ \overline{I}^{2n - (n+k)} =  \overline{I}^{n-k} =  (0, (n-(n-k)+i)_{i=1}^{n-k}) = (0, (k+i)_{i=1}^{n-k})\, .$  Reversion gives $\mathbf{flip} \, \overline{I}^{n - k} =   ((k+i)_{i=n-k}^{1}, 0)$, equalling $((n-j+1)_{j=1}^{n-k}, 0)$ after the change of indices $j=n-k-i+1$. From (\ref{Wronskian2}), case $n+1\le d\le 2n-2$, it follows that $I^{n+k} =   ((n-i+1)_{i=1}^{2n-(n+k)}, 0) = ((n-i+1)_{i=1}^{n-k}, 0)$. The case for odd $d$ is proved in the same way using the last expression in (\ref{conjId}). Thus $I^d = \mathbf{flip} \, \overline{I}^{2n - d}$ for $n+1\le d\le 2n-2\, .$ The other identity is proved analogously.
\end{proof}

From (\ref{Id})-(\ref{Idn}) and Lemma \ref{Revlemma} it follows that the number of elements 
in the index vectors for $0\leq d\leq n-1$ are
\begin{equation*}
| I^d| = | I^{2n-d}| = | \overline{I}^d|= | \overline{I}^{2n-d}| =
2 \left\lfloor\frac{d}{2}\right\rfloor +1\quad \text{and}\quad
 | I^n| = | \overline{I}^n| = 2 \left\lfloor\frac{n}{2}\right\rfloor\, . 
\end{equation*}
The complete index vectors for
variables $a$ and conjugate variables $\overline{a}$ are then
given by the concateneted vectors
\begin{equation*}
\mathcal{I}=(I^0,I^1,...,I^{2n}) \qquad\text{and}\qquad \overline{\mathcal{I}}=(\overline{I}^0, \overline{I}^1,..., \overline{I}^{2n})\, .
\end{equation*}
The total number of elements in
$\mathcal{I}$ and $\overline{\mathcal{I}}$ is $n^2 + n$. Let $\mathcal{J} = (\mathcal{J}_0,\ldots,\mathcal{J}_n)$ be a vector where $\mathcal{J}_i$, $i=1,\ldots,n$, indicates the position of variable $a_i$ in $x(a)$. Also define $\mathcal{J}_0:=0$, $x_0 := a_0$ and 
$\overline{x}_0 := \overline{a}_0$.

\begin{proposition} \label{indexJ}
	The coefficients $a_i$, $i=1,\ldots, n$, of the Blaschke product are located at positions
	\begin{equation}\label{Jindex}
	\mathcal{J}_i=\Big\lceil \frac{(i-1)^2}{2}\Big\rceil +1
	\end{equation}
	in the variable vector $x(a)$, that is $a_i= x_{\mathcal{J}_i}(a)$.
\end{proposition}

\begin{proof}
The proof is by induction for $0\leq d \leq n-1$. Index $i$ corresponds to degree $d+1$. The formula (\ref{Jindex}) holds for $i=1$, since $\mathcal{J}_1=\lceil \frac{(1-1)^2}{2}\rceil +1=1$ and $a_1=x_{\mathcal{J}_1}(a)=x_1(a)$. Now we assume that the formula holds for $i\ge 1$.  If $i$ is even,  then $d=i-1$ is odd and the vector $x^d(a)$ in  (\ref{wddef}) contains $d$ elements. The first element in each vector $x^d(a)$ is the coefficient $a_{d+1}$, so
\begin{equation*}
  \mathcal{J}_{i+1}=\mathcal{J}_i+d=\Big\lceil \frac{(i-1)^2}{2} \Big\rceil +1 + i-1=\frac{(i-1)^2}{2}+\frac{1}{2}+i=\frac{i^2}{2}+1=\Big\lceil \frac{i^2}{2} \Big\rceil +1, 
\end{equation*}
where we in step three used the fact that $i-1$ is odd and in the last step that $i$ is even. If $i$ is odd, then $d=i-1$ is even and the vector $x^d(a)$ contains $d+1$ elements. Then
\begin{equation*}
	\mathcal{J}_{i+1}=\mathcal{J}_i+d+1=\Big\lceil \frac{(i-1)^2}{2} \Big\rceil +1 + i=\frac{(i-1)^2}{2}+1+i=\frac{i^2+1}{2}+1=\Big\lceil \frac{i^2}{2} \Big\rceil +1, 
\end{equation*}
where we used the fact that $i$ is odd in the last step. We conclude that the formula (\ref{Jindex}) holds for all $i=1,\ldots,n$.
\end{proof} 
 \medskip

\noindent{\bf Example 1.}
Index vectors are computed for $n=3$ according to (\ref{Id})-(\ref{Idn}) and
Lemma \ref{Revlemma}. This results in
$I^0=1,\; I^1=2,\; I^2=(3,1,2),\; I^3=(1,3), \; I^4=(3,2,0), \; I^5=I^6=0$ and 
$\overline{I}^0=\overline{I}^1=0,\; \overline{I}^2=(0,2,3),\; \overline{I}^3=(1,3),\; \overline{I}^4=(2,1,3),\; \overline{I}^5=2,\;
\overline{I}^6=1$.
Concatenation gives complete index vecors $\mathcal{I}=(1,2,3,1,2,1,3,3,2,0,0,0)$ and $\overline{\mathcal{I}}=(0,0,0,2,3,3,1,2,1,3,2,1)$.
The length of the variable vector is $n^2+n=12$ and $\mathcal{J}=(0,1,2,3)$,
 where indexing starts from 0. The connection between variable $x_i$ and $a_j\overline{a}_k$ is explicitely given by the products
\begin{equation*}
x_i = a_{\mathcal{I}_i}\,
\overline{a}_{\overline{\mathcal{I}}_i},\quad i=1,...,12\, .
\end{equation*}

After applying (\ref{Jindex}) we obtain $x_i = a_{\mathcal{I}_i}\,\overline{a}_{\overline{\mathcal{I}}_i}=
x_{\mathcal{J}_{\mathcal{I}_i}} \,\overline{x}_{\mathcal{J}_{\overline{\mathcal{I}}_i}}$.
Some examples:
\begin{align*}
x_2 &=x_{\mathcal{J}_{\mathcal{I}_2}} \, \overline{x}_{\mathcal{J}_{\overline{\mathcal{I}}_2}} =
x_{\mathcal{J}_2}\, \overline{x}_{\mathcal{J}_0} = a_2 \overline{a}_0 = a_2\, , \\
x_5 &=x_{\mathcal{J}_{\mathcal{I}_5}} \, \overline{x}_{\mathcal{J}_{\overline{\mathcal{I}}_5}} =
		x_{\mathcal{J}_2}\, \overline{x}_{\mathcal{J}_3} = a_2 \overline{a}_3\, , \\
x_6 &=x_{\mathcal{J}_{\mathcal{I}_6}} \, \overline{x}_{\mathcal{J}_{\overline{\mathcal{I}}_6}} =
		x_{\mathcal{J}_1}\, \overline{x}_{\mathcal{J}_1} = a_1 \overline{a}_1 
		= | a_1 |^2\, , \\
x_{10} &=x_{\mathcal{J}_{\mathcal{I}_{10}}} \, \overline{x}_{\mathcal{J}_{\overline{\mathcal{I}}_{10}}} =
		x_{\mathcal{J}_0}\, \overline{x}_{\mathcal{J}_3} = a_0 \overline{a}_3 = 
		 \overline{a}_3\, . \\
\end{align*}

The next proposition describes the construction of particular
solutions to (\ref{Relaxedsystem}) and to the associated
homogeneous system $A x = 0$.  

\begin{proposition} \label{Propalphabeta}
Let $A x(a) = b$ be the description of the nonlinear system
(\ref{System1}) and
let $A x = b$ be the corresponding relaxed linear system
(\ref{Relaxedsystem}). A particular solution $\alpha$ to
(\ref{Relaxedsystem}) and a non-trivial solution $\beta$ to
$A x = 0$ can be constructed using the fast Fourier
transform and the Wronskian (\ref{blaWronskian1}).
\end{proposition}

\begin{proof} 1. The proof is by construction. First we determine
a particular solution $\alpha$ to (\ref{Relaxedsystem}). We seek
a solution were we put elements in $x = x(a)$ that are of the
form $a_i \overline{a}_{n-j+1}$ in (\ref{System1}) to zero. Thus
 (\ref{System1}) reduces to
\begin{equation}\label{Sysred}
\sum_{i=1}^n i a_i z_k^{i-1} + \sum_{i=1}^n (n-i+1)
\overline{a}_{n-i+1} z_k^{n+i} +
(n+1) z_k^n = 0,
\end{equation}
from which it is clear that we seek a polynomial of degree $2 n$
with its zeros at the $2 n$ distinct critical points and with
coefficient $n+1$ for $z^n$. Define $W_0$ by
\begin{equation*}
W_0(z) := \prod_{j=1}^n (z-z_j) (z-\frac{1}{\overline z_j})\, ,
\end{equation*}
which is the Wronskian in (\ref{blaWronskian1}) with $r=1$ and
$\theta = 0$. Expanding $W_0$ we obtain
\begin{equation*}
  W_0(z) = b_0 + b_1 z +\ \ldots\ + b_{2n-1} z^{2 n - 1} + z^{2n}\, .
\end{equation*}
The coefficients $b_j$ can be computed by sampling $W_0$ on the
unit circle in $2^l$ equidistant points stored in the vector $f =
(W_0(e^{\frac{k 2 \pi i}{2^l}}))_{k=0}^{2^l-1}$. Then the fast
Fourier transform, ($b = \hbox{fft}(f)/2^l$ in Matlab), supplies
us with a vector $b$ containing the coefficients of $W_0$. Thus we
get the correct scaling by defining $c_j = \frac{(n+1) b_j}{b_n}\,
,\ j=0,\ldots, 2 n - 1$ and $c_{2n} = \frac{(n+1)}{b_n}$, giving
the representation
\begin{equation}\label{PolW1}
  W_1(z) = c_0 + c_1 z +\ \ldots\ + c_{2n} z^{2n}\, .
\end{equation}
Then $W_1(z)$ in (\ref{PolW1}) is a self-inversive polynomial
satisfying 
 (\ref{Sysred}) and we determine $a_i$ from the equations $i a_i = c_{i-1},\ i=1,\ldots,n$ and
   put $\alpha := x(a)$. Then $\alpha$ is a particular
   solution to the relaxed linear system $Ax=b$, but not
   to the nonlinear system $A x(a) = b$, since the quadratic
   elements in $x(a)$ were set to zero.

  2. Next we construct a non-trivial solution
   $\beta$ to the homogeneous system $A x = 0$, (the constant term in (\ref{System1}) is omitted).
    Variables in $x = x(a)$ of the form $a_i \overline{a}_{n-j+1}$, $i\ne n-j+1$, in (\ref{System1}) 
    are now put to zero, so that the double sum of the nonlinear terms in (\ref{System1}) is reduced to
 \begin{equation}\label{redsum}
 \Bigl(\sum_{i=1}^n (2 i - n - 1) a_i\overline{a}_{i}\Bigr) z_k^{n}\, .
 \end{equation}
 Notice that if $n$ is odd the middle term is zero in (\ref{redsum}) and the variable $a_i\overline{a}_{i}$ with $i = \left\lceil n/2 \right\rceil$ is omitted from $x(a)$. We select  the nonlinear variables $a_i\overline{a}_{i}$, $i=1,\ldots,\left\lfloor n/2 \right\rfloor,  \left\lceil n/2 \right\rceil +1,\ldots , n$, in $x(a)$ by defining
 \begin{equation}\label{beta_n}
a_i\overline{a}_{i} :=
\begin{cases}
-\frac{3 (n+1-2 i)}{n (n-1)}\, , & i = 1,\ldots, \left\lfloor n/2 \right\rfloor\, ,\\
   - a_{n + 1 - i}\overline{a}_{n + 1 - i}\, ,&  i = \left\lceil n/2 \right\rceil +1,\ldots,n\, .
\end{cases}
\end{equation}
  Then the sum in (\ref{redsum}) is equal to $(n+1)$ and the system (\ref{System1}) with deleted constant term is reduced to the system
 \begin{equation*}
 \sum_{i=1}^n i a_i z_k^{i-1} + \sum_{i=1}^n (n-i+1)
 \overline{a}_{n-i+1} z_k^{n+i} + (n+1) z_k^n = 0\, ,
 \end{equation*}
 which is identical to system (\ref{Sysred}), so we can use the coefficients of the polynomial
 $W_1$ in (\ref{PolW1}) and again determine $a_i$ from $i a_i = c_{i-1},\
 i=1,\ldots,n$. Defining $\beta := x(a)$, we have constructed a nontrivial solution to the homogeneous 
 system $A x = 0$ corresponding to the relaxed linear system (\ref{Relaxedsystem}).
 \end{proof}
\noindent{\bf Remark 1}. A very important observation in the second part of the proof is that the first element
 $\beta_1$ in the solution $\beta$ is always nonzero, (a consequence
 of the fact that all the distinct critical points are nonzero),
 which later on will guarantee that $\beta$ is linearly
 independent of the data independent vectors which together with $\beta$ will form a basis for the null space of the matrix $A$.
 
\noindent{\bf Remark 2}. The particular solution $\alpha$ is self-reversive by construction and the data dependent null space vector $\beta$ is
self-reversive after deletion of elements that correspond to
degree $n$ , (\ref{beta_n}), that is $\alpha = \mathbf{flip}\, \overline{\alpha}$ and 
$\beta \backslash \beta^n  = \mathbf{flip}\, \overline{\beta\backslash \beta^n}$.
\medskip

From now on when using $W(z)$ we always refer to definition (\ref{Wronskian2}).
\medskip

\noindent{\bf Example 2.}
The structure of the quadratic equation system is illustrated for $n=3$. The Ansatz is
\begin{equation*}
B(z) = \frac{a_1 z + a_2 z^2 + a_3 z^3 + z^4}{1 + \overline{a}_3 z + \overline{a}_2 z^2 + \overline{a}_1 z^3}.
\end{equation*}

The Wronskian parts $W^d(z)$ are computed for $d=0$ to $d=6$.
\begin{align*}
& W^0(z) = a_1,\; W^1(z)=2a_2 z,\; W^2(z)=3a_3 z^2 - a_1\overline{a}_2 z^2 + a_2\overline{a}_3 z^2\, , \\
& W^3(z) = -2a_1\overline{a}_1 z^3 + 2a_3\overline{a}_3 z^3 + 4z^3\, ,\\
& W^4(z) = \overline{a}_2 a_3 z^4  - \overline{a}_1 a_2 z^4  + 3\overline{a}_3 z^4,\; W^5(z) = 2\overline{a}_2 z^5,\; W^6(z) = \overline{a}_1 z^{2n}\, .
\end{align*}

The system $W(z_k)=0$, $k=1,\ldots,6$, can be expressed in relaxed linearized
form $Ax=b$, where $x$, with the quadratic structure embedded,
 and $b$ are given by
\begin{align*}
 x &= x(a) = \bigl(a_1, a_2, a_3 , a_1\overline{a}_2, a_2\overline{a}_3, a_1\overline{a}_1, a_3\overline{a}_3,\overline{a}_2 a_3, \overline{a}_1 a_2, \overline{a}_3, \overline{a}_2, \overline{a}_1\bigr)^T\, ,\\
 b &=  -4\, \bigl(z_1^3, z_2^3,z_3^3,z_4^3,z_5^3,z_6^3\bigr)^T\, .
\end{align*}
 The matrix $A$ is of size $6\times 12$ in our relaxed linear system;

\begin{equation*}
        \underbrace{
        \left[
        \begin{array}{cccccccccccc}
            1 & 2z_1 & 3z_1^2 & -z_1^2 & z_1^2 & -2z_1^3 & 2z_1^3 & z_1^4 & -z_1^4 & 3z_1^4 & 2z_1^5 & z_1^6  \\
            \vdots &  &  &  &  &  & \vdots &  &  &  &  & \vdots \\
            1 & 2z_6 & 3z_6^2 & -z_6^2 & z_6^2 & -2z_6^3 & 2z_6^3 & z_6^4 & -z_6^4 & 3z_6^4 & 2z_6^5 & z_6^6
            \end{array} \right]}_{A} x = b\, .
\end{equation*}
Since $z_k$ are distinct points rank $ A = 6$ and $\dim N(A) = 12 - 6 = 6$. A set of nullspace vectors of $A$ is, for example, given by the columns of the matrix

\begin{equation*}
C = 
 \begin{tikzpicture} [baseline=(m.center)]
      \matrix [matrix of math nodes,left delimiter= {[ } ,right delimiter={]^T \, , }] (m)
        {       
            0 & 0 & 1 & 3 & 0 & 0 & 0 & 0 & 0 & 0 & 0 & 0 \\
            0 & 0 & 0 & 1 & 1 & 0 & 0 & 0 & 0 & 0 & 0 & 0 \\
            0 & 0 & 0 & 0 & 0 & 1 & 1 & 0 & 0 & 0 & 0 & 0 \\
            0 & 0 & 0 & 0 & 0 & 0 & 0 & 1 & 1 & 0 & 0 & 0 \\
            0 & 0 & 0 & 0 & 0 & 0 & 0 & 0 & 3 & 1 & 0 & 0 \\
        };
        \draw[color=black] (m-1-3.north west) -- (m-1-5.north east) -- (m-2-5.south east) -- (m-2-3.south west) -- (m-1-3.north west);
         \draw[color=black] (m-3-6.north west) -- (m-3-7.north east) -- (m-3-7.south east) -- (m-3-6.south west) -- (m-3-6.north west);
          \draw[color=black] (m-4-8.north west) -- (m-4-10.north east) -- (m-5-10.south east) -- (m-5-8.south west) -- (m-4-8.north west);
    \end{tikzpicture} 
\end{equation*}

where the blocks correspond to degrees 2, 3 and 4. 
With $\alpha$ and $\beta$ from Proposition \ref{Propalphabeta},  an affine description of
the solution space is $x=\alpha + Ct + \beta t_\beta$, where 
$\alpha, \beta \in \mathbb{C}^{12}$, $C\in \mathbb{R}^{12\times 5}$ and 
$t\in \mathbb{C}^5$ and $t_\beta \in \mathbb{C}$ contain arbitrary complex weights.
\bigskip

A recipe is now given that allows for the construction of a
complete null space matrix $C$ that corresponds to an arbitrary
number $n$ of distinct prescribed critical points $z_k$. The matrix $C$ is built from blocks that
correspond to different degrees. In the following two propositions the null space of vector $w^d$ defined in (\ref{wddef}) is described. We note that the dimension is zero for the null spaces that correspond to the degrees $0, 1, 2n-1$ and $2n$ . A zero column vector with $d$ elements is denoted by $0^d$.
\begin{proposition}\label{propEven}
	Let $d$ be even, $2\leq d\leq 2n-2$, and $w^d$ defined by (\ref{wddef}).
	
	a) For $2\leq d\leq n-1$ and $2\le k\le d$ we define
	\begin{equation*}
		v_1^d = \begin{pmatrix} d-1 \\ d+1 \\ 0^{d-1} \end{pmatrix} \quad\text{and}\quad
			v_k^d= \begin{pmatrix} 0^{k-1} \\ -d+2k-1 \\ d-2k+3 \\ 0^{d-k} \end{pmatrix}\, .	
	\end{equation*}
	The null space of $w^d$ is
	 \begin{equation*}
	 N(w^d) = \mathrm{span} \{ v_1^d,v_2^d,\ldots,v_d^d \}\quad\text{and}\quad \mathrm{dim}\,     
	 N(w^d)=d\, .
	 \end{equation*}
	 
	b) For $n$ even and $2\leq k\leq n$ we define
	\begin{equation*}
			v_k^n= \begin{pmatrix} 0^{k-2} \\ -n+2k-1 \\ n-2k+3 \\ 0^{n-k} \end{pmatrix}\, .	
	\end{equation*}
	The null space of $w^n$ is 
	\begin{equation*}
	N(w^n)=\mathrm{span} \{ v_2^n,\ldots,v_n^n \}\quad\text{and}\quad 
	 \mathrm{dim} \, N(w^n)=n-1\, .
	 \end{equation*}
	 c) For $n+1\leq d\leq 2n-2$ and $1\leq k\leq 2n - d$ we define
	 $ v_k^d = \mathbf{flip} \, v_{2n-d-k+1}^{2n-d}$.	 
	 The null space of $w^d$ is
	 \begin{equation*}
	 N(w^d)=\mathrm{span} \{ v_1^d,v_2^d,\ldots,v_{2n-d}^d \}\quad\text{and}\quad \mathrm{dim}\,     
	 N(w^d)=2n-d\, .
	 \end{equation*}
\end{proposition}

\begin{proof}
a) For even $d$ the vector $w^d$ has $d+1$ elements. Then $w^d\in \mathbb{R}^{d+1}$, so its range has dimension 1 and $\mathrm{dim}\, N(w^d) = d$. The first vector $v_1^d$ belongs to the null space:
\begin{equation*}
	w^d \begin{pmatrix} d-1 \\ d+1 \\ 0^{d-1} \end{pmatrix} = (d+1)(d-1)+(-d+1)(d+1)=0.
\end{equation*}
Next consider vector $v_k^d$ with nonzero elements at positions $k$ and $k+1$.
The first element in $w^d$ is $d+1$ and the elements at positions $k$ and $k+1$ are $-d+2k-3$ and $-d+2k-1$, respectively. Then
\begin{equation*}
	w^d v_k^d = 0+(-d+2k-3)(-d+2k-1)+(-d+2k-1)(d-2k+3)+0=0\, ,
\end{equation*}
so $v_k^d$ is also in the null space. It is clear from the construction that $\{v_1^d,\ldots,v_d^d\}$ is a linearly independent set of vectors. Thus the set spans the null space of $w^d$ for $2\le d\le n-1$. The case b) is proved analogously.

c) Consider the product $w^d v_k^d$ for some $n+1 \leq d \leq 2n-2$ and 
$1\leq k \leq 2n-d$. From case a) it follows that $w^{2n-d} v_k^{2n-d} =0$,  
or  $w^{2n-d} v_{2n-d-k+1}^{2n-d} =0$ if the order of the vectors in 
$N(w^{2n-d})$ is reversed. The inner product is unaffected by 
reversion of both vectors so  $\mathbf{flip}\, w^{2n-d} \mathbf{flip} \, v_k^{2n-d} =0$. Then
\begin{equation*}
w^d v_k^d = w^d \mathbf{flip} \, v_{2n-d-k+1}^{2n-d} =
\mathbf{flip} \, w^{2n-d} \mathbf{flip} \, v_{2n-d-k+1}^{2n-d} = 0\, , 
\end{equation*} 
and vector $v_k^d = \mathbf{flip}\, v_{2n-d-k+1}^{2n-d}$ belongs to $N(w^d)$.
Clearly, by construction, all $v_k^d$ are linearly independent so the nullspace for degree $d$,  $n+1\leq d \leq 2n-2$, is
spanned by vectors $v_k^d$, with $k=1,\ldots , 2n-d$.
\end{proof}

\noindent {\bf Remark 3}. For even $d$ let $C^d$, $2\leq d\leq n-1$, be the matrix with 
columns $v_k^d$ for $k=1,\ldots ,d$. Then dim$\, C^d = (d+1)\times d$, and dim$\, C^n = n\times (n-1)$ for even $d=n$. From case c) it follows for even $d$ with 
$n+1\leq d\leq 2n-2$ that $C^d = \mathbf{flip}\, C^{2n-d}$.
\medskip

\begin{proposition}\label{propOdd}
Let $d$ be odd, $3\leq d\leq 2n-3$, and $w^d$ defined by (\ref{wddef}).
	
	a) For $3\leq d\leq n-1$ and $2\le k\le d-1$ we define
	\begin{equation*}
		\begin{array}{ll}
		v_1^d= \begin{pmatrix} d-1 \\ d+1 \\ 0^{d-2} \end{pmatrix}\quad\text{and}\quad v_k^d=\begin{pmatrix} 0^{k-1} \\ -d+2k-1 \\ d-2k+3 \\ 0^{d-k-1} \end{pmatrix}\, ,\quad  2\leq k \leq \frac{d-1}{2}\, , \\ [2em]

v_{(d+1)/2}^d=\begin{pmatrix} 0^{(d-1)/2} \\ 1 \\ 1 \\ 0^{(d-3)/2} \end{pmatrix} \quad\text{and}\quad	
v_k^d=\begin{pmatrix} 0^{k-1} \\ -d+2k+1 \\ d-2k+1 \\ 0^{d-k-1} \end{pmatrix}\, , \quad\frac{d+3}{2}\leq k \leq d-1\, .
\end{array} 
\end{equation*}
The null space of $w^d$ is
\begin{equation*}
   N(w^d)=\mathrm{span} \{ v_1^d,v_2^d,\ldots,v_{d-1}^d \}\quad\text{and}\quad \mathrm{dim} \, N(w^d)=d-1\, .
\end{equation*}

	b) For $n$ odd and $2\le k\le n-1$ we define
	\begin{equation*}
		\begin{array}{ll}
v_{(n+1)/2}^n= \begin{pmatrix} 0^{(n-3)/2} \\ 1 \\ 1 \\ 0^{(n-3)/2} \end{pmatrix}\quad\text{and}\quad v_k^n=\begin{pmatrix} 0^{k-2} \\ -n+2k-1 \\ n-2k+3 \\ 0^{n-k-1}\, , \end{pmatrix}\quad 2\leq k \leq \frac{n-1}{2}\, , \\ [2.5em]

v_k^n=\begin{pmatrix} 0^{k-2} \\ -n+2k+1 \\ n-2k+1 \\ 0^{n-k-1} \end{pmatrix}\, ,\quad\frac{n+3}{2}\leq k \leq n-1\, .
		\end{array}
	\end{equation*}
	 Then null space of $w^n$ is
	\begin{equation*}
	N(w^n)=\mathrm{span} \{ v_2^n,\ldots,v_{n-1}^n \}\quad\text{and}\quad \mathrm{dim}\, N(w^n)=n-2\, .
	\end{equation*}

c) For $n+1\leq d\leq 2n-3$, $1\le k\le 2n-d-1$, define
$v_k^d = \mathbf{flip} \, v_{2n-d-k+1}^{2n-d}$.	
	The null space of $w^d$ is 
	\begin{equation*}
	N(w^d)=\mathrm{span} \{ v_1^d,v_2^d,\ldots,v_{2n-d-1}^d \}\quad\text{and}\quad \mathrm{dim} \, N(w^d)=2n-d-1\, .
	\end{equation*}
\end{proposition}

\begin{proof}
	a) This case is proved similarly as case a) in Proposition \ref{propEven}.  Again it is clear from the construction that $\{v_1^d,\ldots,v_{d-1}^d\}$ is a linearly independent set of vectors. Now $w^d\in \mathbb{R}^{d}$, so compared to Proposition \ref{propEven}, where $w^d\in \mathbb{R}^{d+1}$, the dimension is reduced by one in all cases.
	The first vector $v_1^d$ and the case when $2\leq k \leq (d-1)/2$ is identical to the even case,  apart from the decrease in dimension by 1, so these vectors belong to the null space.
	The middle case when $k=(d+1)/2$ is slightly different. The element -2 in $w^d$ is at position $(d-1)/2$ and the element 2 is at position $(d+1)/2$, hence it follows that $w^d v_{(d+1)/2}^d = -2\cdot 1 + 2\cdot 1=0$ and this vector belongs to the null space. The last case when $(d+3)/2\leq k \leq d-1$ results in 
\begin{equation*}
	w^d v_k^d = 0+(-d+2k-1)(-d+2k+1)+(-d+2k+1)(d-2k+1)+0=0\, ,
\end{equation*}
so $v_k^d$ belongs to the null space.  The case b), where $d=n$, is proved similarly, only with the first vector missing. The case c) is proved in the same manner as case c) in
	the previous proposition.
\end{proof} 
\medskip

\noindent {\bf Remark 4}. For odd $d$ let $C^d$, $3\leq d\leq n-1$, be the matrix with 
columns $v_k^d$ for $k=1,\ldots ,d-1$. Then dim$\, C^d = d\times (d-1)$ and for 
$n$ odd dim$\, C^n = (n-1)\times (n-2)$. From case c) it follows for odd $d$ with 
$n+1\leq d\leq 2n-2$ that $C^d = \mathbf{flip}\, C^{2n-d}$.
\medskip

A complete sparse description of the affine solution space of the relaxed system $Ax = b$ is then given by
\begin{equation}\label{FullSystem}
\begin{pmatrix} x^0 \\ x^1 \\ x^2 \\ \vdots \\ x^{2n} \end{pmatrix} =
\begin{pmatrix} \alpha^0 \\ \alpha^1 \\ \alpha^2 \\ \vdots \\ \alpha^{2n} \end{pmatrix} +
\begin{pmatrix} 0 	& \ldots 	&  		& 0 		\\
				0	& \ldots 	&		& 0		 	\\
				C^2 & 			&		&			\\
				    & C^3		&		&			\\
				    &   		& \ddots &			\\
				    &  			&		& C^{2n-2}		\\
				  0 	& \ldots 	&  		& 0 		\\
				0	& \ldots 	&		& 0		 	\\
\end{pmatrix} 
\begin{pmatrix} t^2 \\ t^3 \\ \vdots \\ t^{2n-2} \end{pmatrix} + 
\begin{pmatrix} \beta^0 \\ \beta^1 \\ \vdots \\ \beta^{2n} \end{pmatrix} t_\beta\, .
\end{equation}
 A compact representation of this system is $x=\alpha + Ct + \beta t_\beta$, where all nonspecified entries in $C$ are zero. The system can also be expressed degree-wise as
 $x^d = \alpha^d + C^d t^d + \beta^d t_\beta$ for $2\leq d\leq 2n-2$ and 
 $x^d = \alpha^d + \beta^d t_\beta$ for $d=0,1,2n-1,2n$. In the following proposition some properties of this system are summarized.

\begin{proposition}
 The system (\ref{FullSystem}) has the following structural properties:

a) $\alpha ^d = \mathbf{flip}\, \overline{\alpha ^{2n-d}}$, $0\leq d \leq 2n$.

b) $\beta ^d = \mathbf{flip}\, \overline{\beta ^{2n-d}}$, $0\leq d \leq 2n,\, d\neq n$.

c) $x^d = \mathbf{flip}\, \overline{x^{2n-d}}$, $0\leq d \leq 2n,\, d\neq n$.

d) $C^d = \mathbf{flip}\, C^{2n-d}$, $2\leq d \leq 2n-2$.

e) $t^d = \mathbf{flip}\, \overline{t^{2n-d}}$, $2\leq d \leq 2n-2,\, d\neq n$.

f) The variable $t_{\beta}$ is real.
\end{proposition}

\begin{proof}
Cases a), b) and d) are implied by the proofs of Propositions \ref{Propalphabeta}-\ref{propOdd}. Case c) follows from the quadratic structure of the variable vector $x(a)=(x^0(a), \ldots , x^{2n}(a))$ defined by the vectors $x^d(a)$ introduced in (\ref{wddef}). 

Case f). The first element of $x(a)$ is 
$x_1 = a_1 = \alpha_1 + \beta_1 t_{\beta}$ and the last element is
$x_{n^2+n} = \overline{a}_1 = \alpha_{n^2+n} + \beta_{n^2+n} t_{\beta}$.
Clearly $x_1 = \overline{x}_{n^2+n}$ so $\alpha_1 + \beta_1 t_{\beta} =
\overline{\alpha_{n^2+n} + \beta_{n^2+n} t_{\beta}} = \alpha_1 
 + \beta_1 \overline{t_{\beta}}$
and it follows that $\overline{t_{\beta}}= t_{\beta}$ so $t_{\beta}$ is real. 

Case e). We study the equality $x^d - \mathbf{flip}\, \overline{x^{2n-d}}=0$, from case c), for a $d$ such that  $2\leq d \leq 2n-2$.
\begin{equation*} 
\begin{aligned}
x^d - \mathbf{flip}\, \overline{x^{2n-d}} & = \alpha^d + C^d t^d + \beta^d t_{\beta} - 
\mathbf{flip}\, ( \overline{\alpha^{2n-d}} + C^{2n-d} \overline{t^{2n-d}} + \overline{\beta^{2n-d}} t_{\beta} ) \\
& = C^d t^d - \mathbf{flip}\, (C^{2n-d} \overline{t^{2n-d}}) \\
& = C^d t^d - C^d \mathbf{flip}\, \overline{t^{2n-d}}  \\
& = C^d (t^d - \mathbf{flip}\, \overline{t^{2n-d}}) = 0\, ,
\end{aligned}
\end{equation*}
so $t^d = \mathbf{flip}\, \overline{t^{2n-d}}$, since matrix $C^d$ is full rank.
\end{proof}

\begin{proposition}\label{nullspace}
	The null space of matrix $A$, with $\mathrm{dim} \, N(A) = n^2-n$, is completely defined by the structure of the Wronskian $W(z)$, i.e. the matrix $C$, together with the data dependent vector $\beta$ from Proposition \ref{Propalphabeta}.
\end{proposition}

\begin{proof}
The vector $\beta$ belongs to the null space of $A$ by construction. Each column of $C$ belongs to the null space of $A$ since each column of $C^d$ contains
a null space vector $v_k^d$ that corresponds to some degree $d$, $2\leq d\leq 2n-2$. All columns of $C$ are also clearly linearly independent by Propositions \ref{propEven} and \ref{propOdd}. 
Recall that $\mathrm{dim} \, N(w^d)=\mathrm{dim} \, N(w^{d+1})$ for even $d$. Consider the case of even $n$:
\begin{align*}
\sum_{d=2}^{2n-2} \mathrm{dim}\, N(w^d) & = \sum_{d=2}^{n-1} \mathrm{dim} \, N(w^d) + \mathrm{dim}\, N(w^n) +  \sum_{d=n+1}^{2n-2} \mathrm{dim} \, N(w^d)\\
								 & = \sum_{d=2,d\, \mathrm{even}}^{n-1} 2\mathrm{dim} \, N(w^{d}) + n-1 + \sum_{d=n+2,d\, \mathrm{even}}^{2n-2} 2\mathrm{dim} \, N(w^{d}) \\
								 & = \sum_{d=1}^{(n-2)/2} 2\mathrm{dim} \, N(w^{2d}) + n-1 +  \sum_{d=1}^{(n-2)/2} 2\mathrm{dim} \, N(w^{n+2d})\\
								 & = 2\sum_{d=1}^{(n-2)/2} (2d+2n-(n+2d)) + n-1 \\
								 & = n-1+2\sum_{d=1}^{(n-2)/2} n =n(n-2) + n-1=n^2-n-1\, .
\end{align*}
For odd $n$ a similar computation gives $\sum_{d=2}^{2n-2} \mathrm{dim}\, N(w^d) = n^2-n-1$.
Thus $C$ has $n^2-n-1$ linearly independent columns and $\mathrm{rank}\, C=n^2-n-1$.
Since $\beta_1\neq 0$, due to Remark 1, and the first element in each column of $C$ is zero it is clear that the data dependent vector $\beta$ together with all columns of $C$ form a linearly independent set of vectors. Hence, including  $\beta$, a complete description of the null space with $n^2-n$ null space basis vectors is obtained. 
\end{proof}
\medskip

At this point {\bf we define} $m:=\lfloor n^2/2 \rfloor$ and
$p:=\lfloor n^2/2 \rfloor + n$ and the {\bf reduced} Wronskian $\hat{W}(z)$
 \begin{equation*}
 \hat{W}(z) := \sum_{d=0}^{n} W^d(z)\, + (n+1)z^n .
 \end{equation*}
A straight forward calculation gives that $p$ equals the total number 
of terms in $\hat{W}(z)$ that contain variables $a$.
Note that the solution to $Ax(a)=b$ is completely
determined by the values $a_1$ to $a_n$ and the self-reversive
structure of $x(a)$. It is enough to consider the part 
$\hat{x}=\hat{x}(a) = (x^0,x^1,\ldots, x^n)\in \mathbb{C}^p$ of the
vector $x(a)$ that corresponds to degrees $0$ to $n$, since the
rest can be obtained by conjugation and reversion. Thus we define the
\textbf{reduced} system
\begin{equation*}
 \hat{A}\hat{x}(a)=\hat{b}\, ,
\end{equation*}
where $\hat{A}\in \mathbb{C}^{n\times p}$, $\hat{x}(a)\in \mathbb{C}^p$
and $\hat{b} := -(n+1)[z_1^n,...,z_n^n]^T$. The matrix $\hat{A}$ is
the submatrix consisting of rows 1 to $n$ and columns 1 to $p$ of
$A$. For distinct critical points $z_k$ we have rank$\, \hat{A}=n$ and
dim$\, N(\hat{A})=p-n=m$. Any reduced
solution to $\hat{A}\hat{x}(a)=\hat{b}$ can be expanded to a full
solution of $Ax(a)=b$ by concateneting conjugated elements.
Conversely, any full solution $x(a)$ can be reduced to a solution
$\hat{x}(a)$ by deletion of end elements. 
Omitting the quadratic structure, we study the
solution space defined by the \textbf{reduced relaxed linear}
system
\begin{equation*} 
\hat{A}\hat{x}=\hat{b}
\end{equation*}
 defined by the reduced Wronskian 
$\hat{W}(z)$.
 We adopt an identical notation for the reduced particular solution $\hat{\alpha}$ to $\hat{A}\hat{x}=\hat{b}$ and for the reduced data dependent vector $\hat{\beta}$, which is {\bf not} a solution to $\hat{A}\hat{x}=0$.
Let again $C^d$ be the matrix with columns $v_k^d$ constructed according to Propositions \ref{propEven} and \ref{propOdd} and let subvectors in the vectors $\hat{x}$, $\hat{\alpha}$ and $\hat{\beta}$ corresponding to degree $d$ be denoted by $x^d$, $\alpha^d$ and $\beta^d$, $0\leq d\le n$. For weights $t^d$ and $t_\beta$ we define the reduced system
\begin{equation}\label{RedSystem}
\begin{pmatrix} x^0 \\ x^1 \\ x^2 \\ \vdots \\ x^{n} \end{pmatrix} =
\begin{pmatrix} \alpha^0 \\ \alpha^1 \\ \alpha^2 \\ \vdots \\ \alpha^{n} \end{pmatrix} +
\begin{pmatrix} 0 	& \ldots 	&  		& 0 		\\
				0	& \ldots 	&		& 0		 	\\
				C^2 & 			&		&			\\
				    & C^3		&		&			\\
				    &   		& \ddots &			\\
				    &  			&		& C^n		\\
\end{pmatrix} 
\begin{pmatrix} t^2 \\ t^3 \\ \vdots \\ t^n \end{pmatrix} + 
\begin{pmatrix} \beta^0 \\ \beta^1 \\ \vdots \\ \beta^{n} \end{pmatrix} t_\beta
\end{equation}

A compact representation of this system is $\hat{x}=\hat{\alpha} + 
\hat{C}\hat{t} + \hat{\beta} t_\beta$ or 
$x^d = \alpha^d + C^d t^d + \beta^d t_\beta$ for $2\leq d\leq n$ and 
$x^d = \alpha^d + \beta^d t_\beta$ for $d=0,1$.
Block $C^d$ is positioned in the block matrix $\hat{C}$ with its northwest corner at row $\lceil d^2/2\rceil +1$ and column $\lfloor (d-1)^2/2\rfloor +1$. 

\begin{proposition}
a) If the vector $x\in\mathbb{C}^{n^2+n}$ is contained in the affine space (\ref{FullSystem}) then the reduced vector $\hat x\in\mathbb{C}^{p}$ is in the affine space (\ref{RedSystem}).

b) Any vector $\hat x$ in the space (\ref{RedSystem}) can be extended to a vector $x$ contained in the space (\ref{FullSystem}).     
\end{proposition}
\begin{proof}
a) This case is clear from the construction of the affine spaces.

b) Let $(\hat{x}, \hat{t}, t_\beta)$ be a solution to the reduced system, that is
 $x^d = \alpha^d + C^d t^d + \beta^d t_\beta$ for $2\leq d\leq n$ and 
$x^d = \alpha^d + \beta^d t_\beta$ for $d=0,1$. Using the results from Proposition
\ref{nullspace} it is clear that this solution can be extended 
to degrees $n+1\leq d\leq 2n$ according to
$\alpha ^d = \mathbf{flip}\, \overline{\alpha ^{2n-d}}$, 
$\beta ^d = \mathbf{flip}\, \overline{\beta ^{2n-d}}$ and
$x^d = \mathbf{flip}\, \overline{x^{2n-d}}$ and for $n+1\leq d\leq 2n-2$
according to $C^d = \mathbf{flip}\, C^{2n-d}$ and
$t^d = \mathbf{flip}\, \overline{t^{2n-d}}$. If we append
degrees $n+1\leq d\leq 2n$ to the reduced solution 
$(\hat{x}, \hat{t}, t_\beta)$ the obtained triplet $(x,t,t_\beta)$ clearly
represents a solution to the full system (\ref{FullSystem}).
\end{proof}

We summarize our investigation in the main theorem that defines 
the sparse quadratic model which has as solutions a finite Blaschke product and Blaschke forms of degree $n+1$, all of the form (\ref{Ansatz}).
The first set of constraints represents the affine description of the solution set
and the second set imposes the conjugate quadratic structure embedded in $\hat{x}(a)$.

\begin{theorem}
A sparse representation of the quadratic system $\hat{A}\hat{x}(a)=b$ is given by
\begin{equation}\label{sparsemodel}
\begin{cases}
\hat{x}=\hat{\alpha} + \hat{C} \hat{t} + \hat{\beta} t_\beta \\
x_i = x_{\mathcal{J}_{\mathcal{I}_i}} \, 
\overline{x}_{\mathcal{J}_{\overline{\mathcal{I}}_i}},
\quad i\in \{0,1,2,...,p\}\backslash \mathcal{J}\, ,
\end{cases}
\end{equation}
where $t\in \mathbb{C}^{m-1}$, $t_\beta \in \mathbb{R}$ and $\hat{x}\in  \mathbb{C}^p$. A solution to this 
system represents a Blaschke product (or form) and the coefficients $a_i$ can be retrieved
from the solution according to $a_i = x_{\mathcal{J}_i}$ for $i=1,\ldots , n$.
\end{theorem}
\medskip

\noindent\textbf{Remark 5.} This sparse representation contains no high powers of the critical points $z_k$. The only data dependent parts are the reduced particular solution $\hat{\alpha}$ and the reduced data dependent (null space) vector $\hat{\beta}$. The number of linear constraints is $p$, the number of quadratic constraints is $m$ and the total number of variables $p+m$. If we substitute all occurences of variables $x$ in the quadratic constraints, using the sparse affine representation $\hat{x}=\hat{\alpha} + \hat{C} \hat{t} + \hat{\beta} t_\beta$, we arrive at a quadratic square system of size $m\times m$ in variables $t$ and $t_\beta$.
\medskip

\noindent {\bf Example 3.} Let $\hat{\alpha}\in \mathbb{C}^7$ and $\hat{\beta}\in
\mathbb{C}^7$ be known vectors computed according to the recipe in Proposition \ref{Propalphabeta}. The sparse quadratic system for
$n=3$ is given by
\begin{align*}
x_1 = \alpha_1 &       &     &       &+\beta_1 t_\beta \\
x_2 = \alpha_2 &       &     &       &+\beta_2 t_\beta \\
x_3 = \alpha_3 &+t_1   &     &       &+\beta_3 t_\beta \\
x_4 = \alpha_4 &+3t_1  &+t_2 &       &+\beta_4 t_\beta \\
x_5 = \alpha_5 &       &+t_2 &       &+\beta_5 t_\beta \\
x_6 = \alpha_6 &       &     &+t_3   &+\beta_6 t_\beta \\
x_7 = \alpha_7 &       &     &+t_3   &+\beta_7 t_\beta
\end{align*}

\begin{equation*}
x_4=x_1\overline{x}_2,\; x_5=x_2\overline{x}_3,\; x_6=|x_1|^2,\; x_7=|x_3|^2
\end{equation*}
\begin{equation*}
x\in \mathbb{C}^7,\; t\in \mathbb{C}^3,\; t_\beta \in \mathbb{R} .
\end{equation*}

Substitution of $x$ variables gives a $4\times 4$ square quadratic system in variables $t$
and $t_\beta$,
\begin{align*}
\alpha_4 +3t_1 +t_2 +\beta_4 t_\beta &= (\alpha_1 +\beta_1 t_\beta)(\overline{\alpha_2 +\beta_2 t_\beta}) \\
\alpha_5 +3t_1 +t_2 +\beta_5 t_\beta &= (\alpha_2 +\beta_2 t_\beta)(\overline{\alpha_3 + t_1 +\beta_3 t_\beta}) \\
\alpha_6 +t_3 +\beta_6 t_\beta &= (\alpha_1 +\beta_1 t_\beta) (\overline{\alpha_1 +\beta_1 t_\beta}) \\
\alpha_7 + t_3 + \beta_7 t_\beta &= (\alpha_3 +t_1 +\beta_3 t_\beta) (\overline{\alpha_3 +t_1 +\beta_3 t_\beta}).
\end{align*}

From the solution $t$ the coefficients in (\ref{Ansatz}) are calculated as $a_i=x_{\mathcal{J}_i},\, i=1,2,3$
giving
$a_1 = \alpha_1 +\beta_1 t_\beta$, $a_2 = \alpha_2 +\beta_2 t_\beta$ and 
$a_3 = \alpha_3 + t_1 +\beta_3 t_\beta$.
\medskip

\section{Numerical experiments and illustrations}

In this section we conduct a series of numerical experiments to illustrate the efficiency of the proposed method.
All instances were solved using a Matlab R2014a implementation of the method. To solve the quadratic system the command \texttt{fsolve} was used with the following options (optimset): 

\texttt{Algorithm = trust-region-reflective}

\texttt{TolFun = 1e-12}

\texttt{MaxIter = 5000}. 

Different algorithms were initially tested but the trust-region reflective method was the best choice. Internal numerical differentiation was used and the sparsity pattern of the Jacobian was supplied. The reduced particular solution $\hat{\alpha}$ was used as the initial point for \texttt{fsolve} in all cases. The computer used for the test suite was a laptop Intel (R) Core (TM) i5-5200U CPU 2.2 GHz with RAM 8 GB running Windows 7 64-bit. 
A total of $N=50$ random instances are generated and solved for each case. An instance is classified as accurately solved if the solver  stops within 5000 iterations with a tolerance less than 1e-12 and with a maximal error between the prescribed and computed critical points less than 0.5e-4 (4 correct decimals). The  standard norm in $\mathbb{C}$ is used in the error computation. Computational results (iterations, cpu-time, maximal error) are usually reported as triplets consisting of the minimum/median/maximum observation out of $N=50$ instances.

All solutions in this section are checked and they are Blaschke products, although Blaschke forms are also possible solutions to system (\ref{sparsemodel}). The point $x_{\mathrm{initial}}:=\hat{\alpha}$ seems to be a 
very stable choice as the initial point that strongly enhance convergence to
the solution that corresponds to a Blaschke product.
If some other choice of initial point is used, for example zero or any random point, the solution is much more likely to be a Blaschke form with at least one pole in $\mathbb{D}$.
\medskip

\textbf{Optimal assignment of computed critical points}

In order to produce accurate maximal errors between computed and prescribed critical points an assignment problem is solved \cite{BC}. Each computed criticial point $z^c_i,\, i=1,\ldots ,n$, is assigned to a prescribed 
critical point $z_i,\, i=1,\ldots ,n$, in an optimal way. An optimal pairing can be obtained by 
minimizing the maximal distance between every possible pair of computed and prescribed critical points. A distance
matrix $D$ is defined by $d_{ij}=|z_i - z^c_j | $ resulting in the following bottleneck linear assignment problem \cite{BC}, pp. 29-31, with variables $x_{ij}\in \{0,1\}$
\begin{align*}
 \min_{x} \quad & \max_{1\leq i,j \leq n} \quad d_{ij}x_{ij} \\
& \sum_{j=1}^n x_{ij} = 1, \quad i=1,\ldots ,n \\
& \sum_{i=1}^n x_{ij} = 1, \quad j=1,\ldots ,n \\
& x_{ij} \in \{0,1\} ^n.
\end{align*}
The constraint matrix in this problem is totally unimodular so the binary conditions can be relaxed to $x_{ij}\geq 0$. For each problem in this section the optimal pairing is given as
the solution to this linear programming problem. All problems are solved using Matlab2014a and the CVX toolbox \cite{GB}.
\medskip

\textbf{Transformation of data}

During our initial experiments we also observed that it is very beneficial to always transform the data (critical points) using the transformation (2.1) according to
\begin{equation*}
b(z)=\frac{z-z_\star}{1-\overline{z}_\star z }, \quad z_\star=(z_1+\ldots+z_n)/n\in\mathbb{D}.
\end{equation*}

If $z_\star=z_j$ for some $j=1,\ldots,n$ a small random perturbation is added to $z_\star$ avoiding a zero critical point. This transformation will produce an approximate mean centering of the critical points, resulting in problems that very often are solved much more efficiently. 

Let $\tilde B (z)$ be a finite Blaschke product of degree $n+1$ of form (\ref{Ansatz}) with prescribed critical points $ b(z_1),\ldots,b(z_n)$ and computed critical points $\tilde z^c_j$ and zeros $\tilde\alpha_j$. Then the Blaschke product $B_1(z) := \tilde{B}(b(z))$ with critical points $z^c_j := b^{-1}(\tilde z^c_j)$ is a numerical solution of Problem I for prescribed critical points $z_1,\ldots,z_n$. $B_1$ has zeros $\alpha_j = b^{-1}(\tilde\alpha_j)$ but it is not of form (\ref{Ansatz}). To remedy this situation we postcompose $B_1$
 with the automorphism $b_p(z)$ defined by
\begin{equation*}
b_p(z) := \frac{z-\tilde B(b(0))}{1-\overline{\tilde B(b(0))} z }\, ,
\end{equation*}
giving us $B(z):=(b_p\circ\tilde B\circ b) (z)$ of desired form (\ref{Ansatz}) with critical points  $z^c_j$. The
 implemented algorithm delivers the coefficients of $B$. The reported cpu-times include the transformation of the critial points, the generation, including computation of $\hat{\alpha}$ and $\hat{\beta}$, and solution of the equation system, and the inverse transformation and postcomposition of the obtained solution.
\medskip

\begin{figure}
	\centering
	\includegraphics[width=13cm, height=5cm]{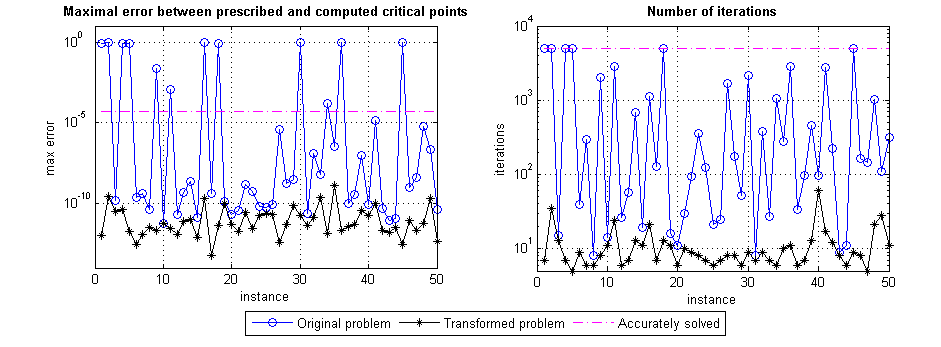}
	\caption{Numerical results when solving original (untransformed) and transformed instances with $n=20$, $r=0.99$ and $N=50$.}
	\label{fig1}
\end{figure}
\medskip

\newpage

\textbf{Numerical test 1: Impact of transformation}

To illustrate the impact of the transformation we generate $N=50$ instances each with $n=20$ critical points randomly and uniformly distributed in a origo centered disk with radius $r=0.99$. The results for the untransformed (original) data is:
\begin{itemize}
	\item cpu-time = 0.25/3.58/111 seconds
	\item iterations = 8/154/5000
	\item max-error = 5.9e-12/1.2e-9/9.5e-1
	\item $76\%$ accurately solved instances (6 instances did not converge within 5000 iterations and another 6 instances did converge but the maximal error exceeded 0.5e-4)
\end{itemize}

For the transformed data the results are significantly improved to:
\begin{itemize}
	\item cpu-time = 0.19/0.30/1.51 seconds
	\item iterations = 5/8/60
	\item  max-error = 6.1e-14/5.0e-12/1.2e-9
	\item  $100 \%$ accurately solved instances.
\end{itemize}

Detailed, instance by instance, results are illustrated in figure \ref{fig1}. The solution time and number of iterations are typically improved by at least a factor 10 and the maximal error is consistenly small
after the transformation is applied. Both the speed of convergence and the accuracy of the obtained solution are significantly improved by the transformation of critical points.
\medskip

\textbf{Numerical test 2: Experiments with different disk radius $r$ }

In this experiment we consider transformed instances with $n=30$ random critical points in disks with the radius ranging from 0.1 to 1. The maximal error and the maximal absolute value of the derivative in computed critical points, $\max_{1\leq i \leq n} |B'(z^c_i)|$, are shown in figure \ref{fig2}. For $0.1\leq r\leq 0.5$ the maximal error is large, but the maximal absolute derivative (as well as the Wronskian) is very small. This indicates that the Blaschke product $B(z)$ is flat almost everywhere in a small disk that contains many critical points. Such instances are in a numerical sense ill-posed since the derivative is small at the same time as the computed critical points can be far from the prescribed points. When $r\geq 0.6$ the maximal error decreases and accuracy is improved. The median number of iterations for each value of the disk radius was 2, 3, 3, 3, 3, 4, 5, 6, 8 and 14.

\begin{figure}
	\centering
	\includegraphics[width=13cm, height=5cm]{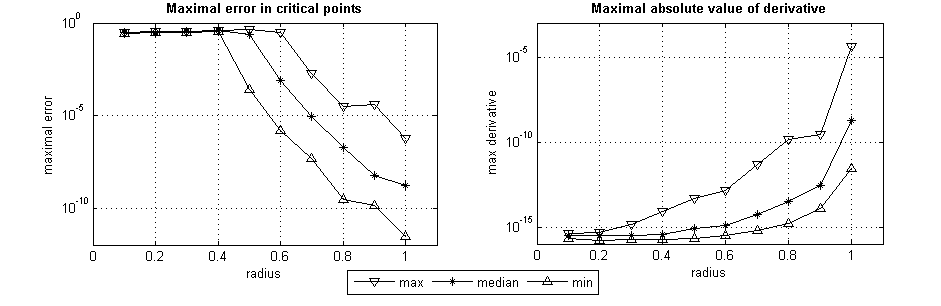}
	\caption{Numerical results when solving transformed instances with $n=30$, $r=0.1,0.2,\ldots ,1$ and $N=50$.}
	\label{fig2}
\end{figure}
\medskip

\textbf{Numerical test 3: Large scale instances}

In this test set large scale instances are considered.
For such instances the transformation is always applied to the critical points prior to solution. 
Random instances with $r=0.999$ and $n$ ranging from 10 to 60 are generated and solved. The results from this test set are presented in table \ref{table1} and figure \ref{fig3}. The results show that random 
instances with critical points in a disk can efficiently and accurately be solved up to $n=40$ and many times even up to $n=50$.
The solver did reach \texttt{MaxIter}=5000 in a few instances and occasionally it reported that the problem was inaccurately solved (typically for $n\geq 50$).
For $n=60$ only a few instances were solved within the prescribed limits. 
If we allow the maximal error to be 0.5e-2 (2 correct decimals), the percentage of accurately solved instances increases to 100\% ($n=40$), 72\% ($n=50$) and 32\% ($n=60$).  The nonlinear system (\ref{sparsemodel}) may be numerically very challenging to solve accurately for more than 50 random critical points in the disk. One could argue that the situation is approaching the one described in Numerical test 2, that is with $n\ge 50$ the sought after Blaschke product $B$ is very flat in the vicinity of some of the prescribed critical points and the problem is approaching numerical ill-posedness. 

\begin{table}[ht!]
	{\scriptsize
		\begin{tabular}{|c|c|c|c|c|c|}
			\hline 
			Points & Size        & Iterations      & Cpu-time (s)         & Maximal error            &  Solved  \\ 
			$n$    & $n^2+n$     & min/median/max  & min/median/max       & min/median/max           &   $\%$   \\ 
			\hline 
			10 & 110              & 4/6/58      & 0.06/0.11/0.56          &  3.4e-15/1.2e-13/6.2e-11 & 100      \\ 
			\hline 
			20 & 420              & 5/11/892    & 0.22/0.37/19.4          &  2.9e-13/9.2e-12/8.3e-7  & 100     \\ 
			\hline 
			30 & 930              & 7/12/489    & 0.61/0.89/23.4 		  & 3.6e-12/1.1e-9/1.8e-6    & 100   \\ 
			\hline 
			40 & 1640             & 7/34/3420   & 1.31/3.72/319           &  5.3e-10/9.6e-7/2.2e-4   & 92    \\ 
			\hline 
			50 & 2550             & 8/40/5000   & 2.11/7.77/879           &  2.8e-7/4.7e-4/9.8e-1    & 30    \\ 
			\hline 
			60 & 3660 			  & 9/83/5000  & 3.85/23.9/1305           &  7.5e-6/4.1e-1/9.8e-1    & 4 \\ 
			\hline 
	\end{tabular} }
	\medskip
	\caption{Numerical results from the test set with large scale instances. The second column corresponds to the size (number of real variables) of system (\ref{sparsemodel}).}
	\label{table1}
\end{table}

\begin{figure}
	\centering
	\includegraphics[width=13cm, height=9.5cm]{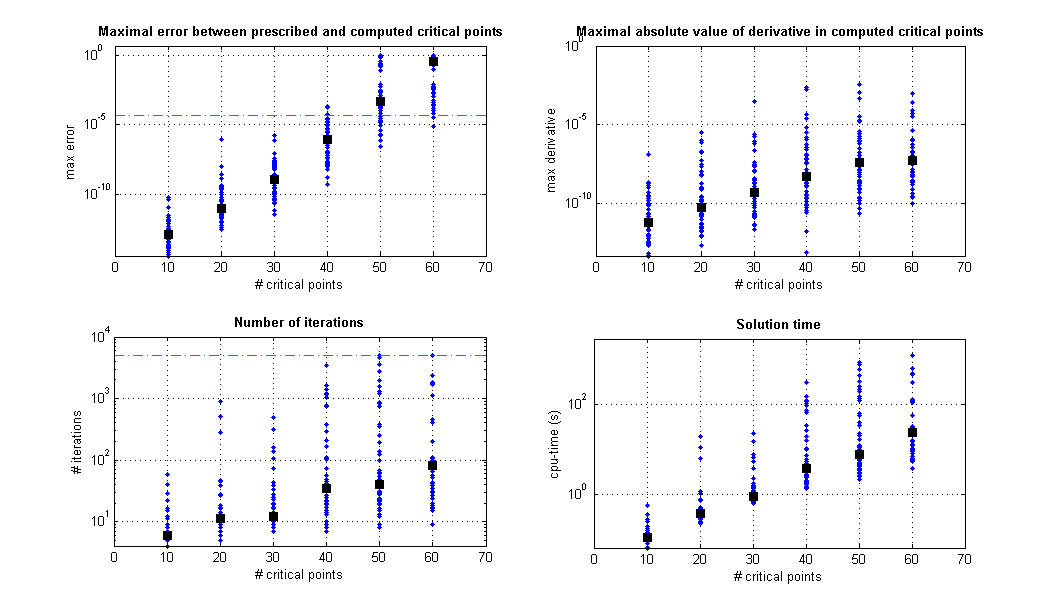}
	\caption{Test results when solving transformed instances with $n=10$ to $n=60$, $r=0.999$ and $N=50$. The dash-dotted lines correspond to the maximal error 0.5e-4 and iteration limit 5000. Each square marks the median result and the dots correspond to separate instances.}
	\label{fig3}
\end{figure}
\bigskip

\textbf{Numerical test 4: Example of hard and easy instances}

As a last experiment we consider one special case which is very hard to solve untransformed and one case that is very easy to solve untransformed. One example of hard problems is when the critical points form a nonorigo centered cluster. Such an instance can be generated by sampling points 
	\begin{equation*}
		z_j=(1+i)/3+1/4(N(0,1)+iN(0,1))
	\end{equation*}
that form a random cluster with midpoint approximately at $(1+i)/3$. If $|z_j|\geq1$ the point $z_j$ is discarded and a new point generated. To illustrate the impact of the transformation we generate $N=50$ instances each with $n=10$ critical points in a cluster. The results for the untransformed (original) data are:
\begin{itemize}
	\item cpu-time = 0.31/16.2/21.6 seconds
	\item iterations = 27/4057/5000
	\item max-error = 5.0e-11/7.5e-1/9.2e-1
	\item $32\%$ accurately solved instances
\end{itemize}
The convergence is very slow for untransformed clustered instances and the majority are not accurately solved. 
\newpage
For transformed data the results are:
\begin{itemize}
	\item cpu-time = 0.03/0.08/0.22 seconds
	\item iterations = 3/4/15
	\item  max-error = 1.9e-14/1.5e-12/7.1e-10
	\item  $100 \%$ accurately solved instances.
\end{itemize}
The improvement in speed of convergence and accuracy of the obtained solution is remarkable when clustered instances are transformed. One type of easy problems can be generated by sampling critical points approximately equidistant (with a small random perturbation) on an origo centered circle. These instances are generated with parameters $r=0.95$, $n=50$ and $N=50$.
The results for untransformed critical points are:
\begin{itemize}
	\item cpu-time = 3.06/4.37/22.5 seconds
	\item iterations = 12/21/132
	\item max-error = 8.6e-13/9.8e-8/2.6e-3
	\item $74\%$ accurately solved instances
\end{itemize}
and for the transformed case the results are:
\begin{itemize}
	\item cpu-time = 2.93/4.20/10.4 seconds
	\item iterations = 11/19/55
	\item  max-error = 1.8e-12/4.9e-9/1.4e-3
	\item  $74 \%$ accurately solved instances.
\end{itemize}

The results are similar. In this case there is only minor impact of the transformation.
The configuration of critical points is in this case almost invariant under the transformation. Instances with $n=50$ random points in a disk are very hard to solve untransformed but instances with an equal number of critical points approximately on a circle are significantly easier solved untransformed as well as transformed. If a maximal error of 0.5e-2 is allowed, all instances are classified as accurately solved.

\section{Discussion}
We have presented a constructive method for determining a finite Blaschke
product of degree $n+1$ having $n$ preassigned distinct critical
points $z_1,\dots,z_n$ in the complex (open) unit disk
$\mathbb{D}$. 

Starting from a dense, highly data dependent and ill-conditioned quadratic system derived from the Wronskian 
$W(z)$, a sparse model that includes a set of affine constraints
and a set of simple quadratic constraints was constructed. The affine part represents a complete description of the 
null space of the Wronskian matrix together with a particular solution to a relaxed linear system $Ax=b$. The null space is almost entirely described by the strucure of the Wronskian and
the data dependency is isolated to one single null space basis vector.
The other data dependent part is the particular solution, and both can be efficiently and reliably computed using the fast Fourier transform.

The numerical model was tested and several experiments showed that random generated instances could very accurately be solved up to $n=40$ critical points and in many cases up to $n=50$.  For $n > 50$ it is likely that any proposed method based on the Wronskian will encounter numerical difficulties due to the increasing numerical ill-posedness of the problem. 

The key components in the proposed method that enables accurate and efficient solution of this demanding problem are: transformation of the critical points, efficient computation of the data dependent particular solution and null space vector, the sparse structure of the system and a stable choice of initial point for the solver. 

Future work could include the possibility of prescribing nondistinct critical points in the disk.


\begin{thebibliography}{SWB}

\bibitem{H} M.\,Heins,
On a class of conformal metrics,
{\it Nagoya Math. J.} {\bf 21} (1962) 1--60.

\bibitem{KR1} D.\,Kraus and O.\,Roth,
Critical points of inner functions, nonlinear partial differential equations, and an extension of Liouville's theorem,
{\it J. London math. Soc.} {\bf 77} (2008) 183--202.

\bibitem{KR2} D.\,Kraus and O.\,Roth,
Critical points, the Gauss curvature equation and Blaschke
products, {\it Blaschke products and their applications, Fields
Inst. Commun.} {\bf 65} (2013) 133--157.

\bibitem{Sch} I.\,Scherbak,
Rational functions with prescribed critical points,
{\it GAFA, Geom. funct. anal.} {\bf 12} (2002) 1365--1380.

\bibitem{ShS} T.\,Sheil-Small,
Complex Polynomials,
{\it Cambridge University Press} 2002.

\bibitem{Sin} D.A.\,Singer,
The location of critical points of finite Blaschke products,
{\it Conform. Geom. Dyn.} {\bf 10} (2006) 117--124.

\bibitem{Ste} K.\,Stephensson,
Introduction to Circle Packing: The Theory of Discrete Analytic Functions,
{\it Cambridge University Press} 2005.


\bibitem{W1} J.L.\,Walsh,
Note on the location of zeros of the derivative of a rational function whose zeros and poles are symmetric in a circle, {\it Bull. Amer. Math. Soc.} {\bf 45} (1939) 462--470.

\bibitem{W2} J.L.\,Walsh,
The location of critical points of analytic and harmonic functions,
{\it AMS Colloquim Publications XXXIV} 1950.

\bibitem{W3} J.L.\,Walsh,
Note on the location of zeros of extremal polynomials in the non-euclidean plane, {\it Acad. Serbe Sci. Publ. Inst. Math.} {\bf 4} (1952) 157--160.

\bibitem{SW} G.\,Semmler and E.\,Wegert,
Finite Blaschke products with prescribed critical points, Stieltjes polynomials, and moment problems,
{\it Anal.Math.Phys,}  DOI 10.1007/s13324-017-0193-5 (2017) 1--29.


\bibitem{Z} S.\,Zakeri,
On critical points of proper holomorphic maps on the unit disk,
{\it Bull. London Math. Soc.} {\bf 30} (1998) 62--66.

\bibitem{BC} R.E.\,Burkard and E. \c{C}ela,
Linear assignment problems and extensions,
In: Du DZ., Pardalos P.M. (eds) Handbook of Combinatorial Optimization.
{\it Springer} (1999), 75--149.


\bibitem{GB} M. Grant and S. Boyd,
CVX: Matlab software for disciplined convex programming, version 2.0 beta. http://cvxr.com/cvx, 2013.


\end{thebibliography}
\end{document}